\documentclass[sn-basic,Numbered]{sn-jnl}%

\usepackage{apxproof}
\usepackage{graphicx}%
\usepackage{enumerate}%
\usepackage{multirow}%
\usepackage{bbm}%
\usepackage{amsmath,amssymb,amsfonts}%
\usepackage{amsthm}%
\usepackage{mathrsfs}%
\usepackage[title]{appendix}%
\usepackage{xcolor}%
\usepackage{textcomp}%
\usepackage{manyfoot}%
\usepackage{booktabs}%
\usepackage{algorithm}%
\usepackage{algpseudocode}%
\usepackage{listings}%
\usepackage{caption}
\usepackage{subcaption}
\usepackage{mathtools}
\usepackage{tabularx}

\theoremstyle{thmstyleone}%
\newtheoremrep{theorem}{Theorem}%
\newtheorem{corollary}{Corollary}[theorem]
\newtheoremrep{lemma}{Lemma}%
\newtheoremrep{proposition}[theorem]{Proposition}%

\theoremstyle{thmstyletwo}%

\theoremstyle{thmstylethree}%

\newcommand{\R}{\mathbb R}

\newcommand{\Sp}{\mathbb{S}_+}
\newcommand{\I}{\mathrm{Id}}
\newcommand{\1}{\mathbbm{1}}
\newcommand{\norm}[1]{\left\lVert#1\right\rVert}
\newcommand{\HH}{\mathcal{H}}

\renewcommand{\v}[1]{\mathbf{#1}}
\DeclareMathOperator{\prox}{prox}
\DeclareMathOperator{\diag}{diag}
\DeclareMathOperator{\nullspace}{null}

\DeclareMathOperator{\vvec}{vec}

\DeclareMathOperator*{\argmin}{argmin}
\DeclareMathOperator*{\minimize}{minimize}

\newenvironment{notation*}
  {\par\vspace{\abovedisplayskip}\noindent
   \tabularx{\columnwidth}{>{$}l<{$} @{${}:{}$} >{\raggedright\arraybackslash}X}}
  {\endtabularx\par\vspace{\belowdisplayskip}}

\begin{document}

\title[Decentralized Sensor Network Localization using Matrix-Parametrized Proximal Splittings]{Decentralized Sensor Network Localization using Matrix-Parametrized Proximal Splittings}

\author*{\fnm{Peter} \sur{Barkley}}\email{peter.barkley@nps.edu}
\author{\fnm{Robert L} \sur{Bassett}}\email{robert.bassett@nps.edu}

\affil{\orgdiv{Operations Research Department}, \orgname{Naval Postgraduate School}, \orgaddress{\street{1 University Circle}, \city{Monterey}, \postcode{93943}, \state{CA}, \country{USA}}}

\abstract{We present a novel application of a recently-proposed matrix-parametrized proximal splitting method to sensor network localization, the problem of estimating the locations of a set of sensors using only noisy pairwise distance information between the sensors.
The decentralized computation required by our approach respects the communication structure between sensors specified by the noisy SNL problem, thereby allowing individual sensors to estimate their location using only local computations and communication with their neighbors. 
Our proposed method experimentally outperforms a competing method for decentralized computation---the alternating direction method of multipliers (ADMM)---with respect to convergence rate and memory use. 
As an independent methodological contribution, we propose using the Sinkhorn-Knopp algorithm in a completely decentralized manner to construct the matrices which parametrize our proposed splitting method. 
We show that parameters selected using this method perform similarly to those selected via existing parameter selection methods while requiring far less computation.
Unlike centralized interior point solution methods, our first order splitting method allows for efficient warm starting, and we demonstrate improvements in convergence using rough estimates of sensor location to warm start our algorithm.
We also find that early termination of the algorithm provides more accurate location estimates than the minimizer of the node-based SDP relaxation of the SNL.
}

\maketitle

\section{Introduction and Preliminaries}\label{Sec:Preliminaries}

This paper presents a new approach for solving the decentralized sensor network localization (SNL) problem using a matrix-parametrized proximal splitting method.
Building on work in \cite{bassett2024optimaldesignresolventsplitting}, we construct a decentralized proximal splitting method which adheres to the sensor network communication structure. 
We apply this method to solve the node-based semidefinite program (SDP) relaxation of the SNL developed in \cite{wang2006further}.
We compare our contributions to a decentralized Alternating Direction Method of Multipliers (ADMM) approach.
In our numerical experiments, we show that ADMM is, across all iterations, on average twice the distance from the true location as our proposed parametrized proximal splitting method.
We also find in our experiments that the solution of our splitting method under an early termination criteria tends to be closer to the true location than the actual solution of the node-based relaxation we reach in the limit.

We introduce a novel splitting matrix design approach using the Sinkhorn-Knopp algorithm \cite{sinkhorn1967concerning} which extends the hitherto-centralized matrix design approach for these proximal splitting algorithms to allow decentralized parameter selection.
This approach performs comparably with the existing SDP-optimized matrix design in practice while taking far less time to execute for large problems.
We prove that the Sinkhorn-Knopp algorithm converges to a valid matrix parameter for the proximal splitting under general conditions.
Additionally, when the communication structure between the sensors is 2-Block, which was shown in \cite{bassett2024optimaldesignresolventsplitting} to achieve both a high degree of parallelism and a fast convergence rate, we prove that the connectedness of the communication network is the only requirement for convergence to a valid matrix parameter.

\subsection{Sensor Network Localization}
Wireless sensor networks have a wide range of military, commercial, and scientific applications which rely on accurate geolocation. 
Geolocation systems like GPS remain unavailable in many contexts, whether due to impermeable environments underwater or below ground, jamming or spoofing in military contexts, or size, weight, and power requirements or costs exceeding the capacity of micro- and nano-sensors. 
Localization based on potentially noisy distance information therefore remains an ongoing area of research, particularly for aerial and maritime unmanned systems and other dispersed terrestrial and maritime sensors \cite{patwari2005locating,biswas2004semidefinite,wang2006further,zhu2018admm}.

The sensor network localization (SNL) problem considered here, as described in \cite{biswas2004semidefinite}, 
consists of estimating the positions of a set of $n$ sensors given only the locations of a small set of anchors and a limited set of pairwise distance information between each sensor and the set of sensors and anchors within some given neighborhood of it.
This is also known as a graph realization problem, in which a set of nodes, partitioned into sensors and anchors, is localized in a Euclidean space of dimension $d$ in such a way that the distances between nodes matches the given distance information.
When the distance information contains noise, we distinguish the problem by referring to it as the noisy SNL. 
Localization in this context then consists of determining a set of sensor locations in $\R^d$ which minimize the deviation from the given distance information under a given norm.

Because of its importance in many application areas, a variety of methods have been created to solve to SNL problem. 
Multidimensional scaling (MDS) provides one approach to solving the noisy SNL.
MDS requires the estimation of the entire Euclidean Distance Matrix (EDM), in which each entry gives $d^2_{ij}$, the distance between sensor $i$ and sensor $j$, rather than just information between sensors in communication range \cite{costa2005achieving,shang2003localization,zhu2018admm}. 
SDP approaches developed in \cite{doherty2001convex,biswas2004semidefinite,biswas2006semidefinite,wang2006further} provide a convex relaxation of the problem which supports a wide range of implementations.
These SDP approaches rely only on the available distances and facilitate a natural extension to the case of noisy distance information.
Both MDS- and SDP-based relaxations allow decentralized approaches. 
These including clustering-based approaches \cite{biswas2006distributed}, subproblems which fix the neighbor locations based on their most recent solutions  \cite{wan2019distributed,shi2010distributed}, and other MDS-based approaches \cite{costa2005achieving,zhu2015distributed}. 
These are, however, sensitive to anchor density and geometry (for MDS and clustering), or rely on differentiable objectives (in the sequential minimization neighbor-fixing approach in \cite{shi2010distributed}).
The node- and edge-based relaxations of the SDP, found in \cite{wang2006further}, provide particularly scalable approaches by reducing the size of the matrices subject to positive semi-definite constraints.

\subsection{Notation and Problem Definition}
We use $\Sp^n$ to denote the set of $n \times n$ symmetric PSD matrices, and $Z \succeq W$ is equivalent to $Z-W \in \Sp^n$.
Unless otherwise noted, $||\cdot||$ refers to the Euclidean norm.
$\I$ is the identity operator and $\1$ is the ones vector in the appropriate dimension.
$\partial$ denotes the subdifferential operator and $\otimes$ the Kronecker product.
We use the notation $G(A)$ to refer to the unweighted graph over the adjacency matrix $A \in \{0,1\}^{n \times n}$, and $G_W(B)$ refers to the weighted graph over the weighted adjacency matrix $B \in \R^{n \times n}$.

The noisy SNL problem data and parameters are defined as follows:
\begin{notation*}
    d \in \mathbb{N} & dimension of the Euclidean space containing the sensors\\
    n \in \mathbb{N} & number of sensors \\
    m \in \mathbb{N} & number of anchors \\
    \mathcal{M}_i & subset of anchors $k$ for which sensor $i$ received a distance observation\\
    \mathcal{N}_i & subset of sensors $j$ for which sensor $i$ received a distance observation\\
    a_k \in \R^d & location of anchor $k$ \\
    d_{ik} \in \R_+ & sensor $i$'s noisy measurement of distance to anchor $k \in \mathcal{M}_i$ \\
    d_{ij} \in \R_+ & sensor $i$'s noisy measurement of distance to sensor $j \in \mathcal{N}_i$ \\
    X \in \R^{n \times d} & decision variable in which row $i$ gives the location estimate of sensor $i$
\end{notation*}
\noindent We refer to $\mathcal{N}_i$ as the neighborhood of sensor $i$, and $j \in \mathcal{N}_i$ as the neighbors of sensor $i$. 
The noisy SNL problem is then defined as:
\begin{align}
    \label{nonconvex_snl}
    \minimize_{X \in \R^{n \times d}} \sum_{i=1}^n \left(\sum_{j \in \mathcal{N}_i} \left| d_{ij}^2 - \norm{X_{i\cdot} - X_{j\cdot}}^2\right|  +  \sum_{k \in \mathcal{M}_i} \left| d_{ik}^2 - \norm{X_{i\cdot} - a_k}^2 \right| \right)
\end{align}
This non-convex problem can then be relaxed into a convex SDP as in \cite{biswas2004semidefinite}. 
For $Y \in \Sp^n$ we define $S: \R^{n \times d} \times \Sp^n \to \Sp^{d+n}$ and $g_i: \R^{n \times d} \times \Sp^n \to \R$ for $i=1, \dots, n$ as
\begin{align}
    S(X,Y) &= 
        \begin{bmatrix}
        \I & X^T \\
        X & Y
        \end{bmatrix} \\
    g_i(X, Y) &= \sum_{j \in \mathcal{N}_i} \left| d_{ij}^2 - Y_{ii} - Y_{jj} + 2Y_{ij}\right| + \sum_{k \in \mathcal{M}_i} \left| d_{ik}^2 - Y_{ii} - \norm{a_k}^2 + 2 a_k^T X_{i \cdot} \right| \label{g_def}
\end{align}
The SDP relaxation of \eqref{nonconvex_snl} is then given by:
\begin{subequations}
\label{convex_snl_sdp}
\begin{align}
    \minimize_{X, Y}  \quad &\sum_{i=1}^n g_i(X,Y)\\
    \text{s.t.} \quad & S(X,Y) \succeq 0
\end{align}
\end{subequations}
Following \cite{wang2006further}, we then further relax \eqref{convex_snl_sdp} into an SDP which enforces the PSD constraint only on specific principal submatrices of $S$.
We let $S^i \in \Sp^{d+1+|\mathcal{N}_i|}$ denote the principal submatrix of $S$ over columns and rows in $\{1, \dots, d$, $i+d\} \cup \{j+d : j \in \mathcal{N}_i\}$, so that the principal submatrix $S^i$ contains only the information on sensor $i$ and its neighbors in $\mathcal{N}_i$. %
The node-based SDP relaxation is then given as:
\begin{subequations}
\label{convex_snl_nsdp}
\begin{align}
    \minimize_{X, Y} \quad &\sum_{i=1}^n g_i(X,Y)\\
    \text{s.t.} \quad & S^i(X,Y) \succeq 0 \quad  i = 1, \dots, n
\end{align}
\end{subequations}
Relaxation \eqref{convex_snl_nsdp} replaces the single large semidefinite cone constraint over the entire matrix $S$ in \eqref{convex_snl_sdp} with $n$ smaller conic constraints.
In so doing, it permits a splitting over the $n$ nodes using only the information relevant within a given neighborhood, whereas the original relaxation in \eqref{convex_snl_sdp} requires communicating the entire matrix to a single node to ensure the full semidefinite constraint is satisfied.
The analysis in \cite{wang2006further} finds that relaxation \eqref{convex_snl_nsdp} is stronger than both an edge-based SDP relaxation and the second-order conic programming relaxation of \eqref{nonconvex_snl}, with solution set equality in the zero noise case whenever the graph over the distance observations is chordal.

\subsection{Matrix-Parametrized Proximal Splitting}
We propose to solve the node-based SDP relaxation \eqref{convex_snl_nsdp} using the class of matrix-parametrized proximal splitting methods developed in \cite{tam2023frugal, bassett2024optimaldesignresolventsplitting}, which are closely related to those in \cite{bredies2024graph}.
These proximal splitting methods leverage results in \cite{ryu2020uniqueness,aragon2023primal, malitsky2023resolvent} to minimize a sum of lower semi-continuous, convex, and proper functions by evaluating each proximal function once in each iteration.
We set the structure of these evaluations by choosing two PSD matrix parameters which determine the algorithm.
These algorithms are implemented in the OARS repository \cite{oars}.
We make use of the repository's matrix parameter generation tools, its serial implementation of the splitting algorithm with callbacks, and its decentralized MPI implementation of the splitting algorithm.
The matrix-parametrized proximal splitting algorithm solves problems of the following form:
\begin{equation}\label{min_sum_func}
    \min_{x \in \HH} \; \sum_{i=1}^{p} f_{i}(x)
\end{equation}
where for some integer $p \geq 2$ and $i = 1, \dots, p$, each $f_i$ is a proper lower semi-continuous convex function such that $f_{i}: \HH \to \mathbb{R} \cup \{\infty\}$. 
The Hilbert space $\HH$ contains the problem's decision variable. 
We further assume the relative interiors of the domain of each $f_{i}$ have nonempty intersection and a minimizer exists.
We employ proximal operators $\prox_{f_{1}}, \dots, \prox_{f_{p}}$, where
\begin{equation}\label{Def:Prox}
\prox_{f_i}(x) = \argmin_{w \in \HH} \; f_{i}(w) + \frac{1}{2} \|w - x\|^{2}.
\end{equation}

We define a \textit{lifted} variable $\v{x} \in \HH^p$ as $\v{x} = (x_1, x_2, \dots, x_p)$ where each $x_i \in \HH$, and, following the notation in \cite{malitsky2023resolvent,tam2023frugal,bassett2024optimaldesignresolventsplitting}, given a matrix $M \in \R^{d \times p}$, define a lifted linear operator from $\HH^p$ to $\HH^d$ as $\v{M} = M \otimes \I$ where $\I$ is the identity operator in $\HH$. 
We also define a lifted operator $\v{F}$ on $\HH^p$ as $\v{F}(\v{x}) = (\partial f_1(x_1), \partial f_2(x_2), \dots, \partial f_p(x_p))$. 

The parametrization of the splitting algorithm relies on the choice of $Z, W \in \Sp^p$ satisfying \eqref{oars}:
\begin{subequations} \label{oars}
\begin{align}
Z \succeq W \label{Z_succeq_W}\\
\nullspace{W} = \text{span}(\1) \label{null_W_1}\\
\diag(Z)=2\1 \label{diag_Z_2}\\
\1^T Z \1 = 0 \label{sum_Z_1}
\end{align}
\end{subequations}
As noted in \cite{bassett2024optimaldesignresolventsplitting}, $Z$ and $W$ can be interpreted as the weighted graph Laplacians over the graphs of the communications required by the algorithm.
We define $L$ as the strictly lower triangular solution to $2\I - L - L^T = Z$.
$\v{L}$ and $\v{W}$ are then lifted operators on $\HH^n$. 
For positive scaling parameter $\alpha$ and step size $\gamma \in (0,1)$, the algorithm iteration is given by:
\begin{subequations}\label{n_iteration_prox}
    \begin{align}
    x_1 &= \prox_{\alpha f_1}\left(v_1^{k} \right)\\
    \vdots &  \nonumber\\
    x_i &= \prox_{\alpha f_i}\left(v_i^{k} + \sum_{j=1}^{i-1} L_{ij}x_j\right)\\
    \vdots & \nonumber\\
    x_p &= \prox_{\alpha f_p}\left(v_p^{k} + \sum_{j=1}^{p-1} L_{pj}x_j\right)\\
    \v{v}^{k+1} &= \v{v}^{k} - \gamma \v{W} \v{x}.\label{n_itr2_}
    \end{align}
\end{subequations}
With $J_{\alpha \v{F}}(\v{y}) = \left(\prox_{\alpha f_1}(y_1), \dots, \prox_{\alpha f_p}(y_n)\right)$, we write \eqref{n_iteration_prox} succinctly in Algorithm \ref{n_iteration}.
At convergence with values $\v{x}^*$ and $\v{v}^*$, we have $\v{x}^* = \1 \otimes x^* = (x^*, \dots, x^*)$ where $x^*$ is a solution to \eqref{min_sum_func}, and for $\v{y}^* = \v{v}^* + (\v{L} - \I)\v{x}^*$, we have $y_i^* \in \partial f_i(x^*)$ and $\sum_{i=1}^n y_i^* = 0$.

\begin{algorithm} 
    \caption{Matrix-Parametrized Proximal Splitting Algorithm}\label{n_iteration} 
    \begin{algorithmic}[1] 
    \Require $\alpha > 0$; $\gamma \in (0,1)$; $\v{v}^0 \in \HH^p$ such that $\sum_{i=1}^p v_i^0 = 0$; $Z, W \in \Sp^p$ satisfying \eqref{oars}; lower triangular $L$ such that $Z=2\I - L - L^T$.
    \State $k \gets 0$
    \Repeat
      \State Solve $\v{x}^k = J_{\alpha \v{F}}\left(\v{v}^{k} + \v{L} \v{x}^k\right)$ for $\v{x}^k$\label{n_itr1}
      \State $\v{v}^{k+1} = \v{v}^{k} - \gamma \v{W} \v{x}^k $\label{n_itr2}
      \State $k \gets k+1$
    \Until{convergence}
\end{algorithmic}
\end{algorithm}

In a distributed computational setting where each proximal function is computed on a separate computation node, no direct communication is required between node $i$ and node $j$ within an iteration when $Z_{ij} = Z_{ji} = 0$.
Similarly, setting $W_{ij} = W_{ji} = 0$ implies that $v_i$ can be computed without $x_j$, and $v_j$ can be computed without $x_i$, in \eqref{n_itr2_}.
Therefore, if $Z_{ij} = W_{ij} = 0$, nodes $i$ and $j$ do not need to communicate directly at any point in Algorithm \ref{n_iteration}.
Given some communication structure described by an adjacency matrix $A$ and associated graph $G(A)$, in which nodes can only communicate if edges exist between them, we can therefore use the sparsity pattern of $W$ and $Z$ to design an algorithm which only communicates via valid edges.
If $Z_{ij} = W_{ij} = 0$ wherever $A_{ij} = 0$, we say that $W$ and $Z$ \textit{adhere to the sparsity structure} in $G(A)$.

One of the most efficient families of matrix parameters for Algorithm \ref{n_iteration} is the 2-Block design \cite{bassett2024optimaldesignresolventsplitting}, which can be used when $p$ is even. 
This design evenly partitions the functions into two sets, and prohibits communication within each set so that the proximal functions can run in parallel.
This is done by choosing $Z$ such that
\begin{equation}\label{2-block}
    Z = 2\begin{bmatrix}
        \I & -B^T \\
        -B & \I
    \end{bmatrix}
\end{equation}
where $B \in \R^{\frac{p}{2} \times \frac{p}{2}}$, adding no additional restrictions on $W$ beyond those in \eqref{oars}.
The parallelism of this design allows for the decentralized computation of the $p$ proximal functions over $\frac{p}{2}$ computation nodes.
Given some adjacency matrix $A \in \{0,1\}^{\frac{p}{2} \times \frac{p}{2}}$ and associated graph $G(A)$ which describe a sparsity pattern for the blocks in the 2-Block matrix parameters, we say $Z$ and $W$ \textit{2-Block adhere} to the sparsity structure in $G(A)$ if for each $i \neq j$ where $A_{ij} = 0$ we have $B_{ij} = B_{ji} = 0$ and $W_{ij} = W_{j+n, i} = W_{i+n, j} = W_{i+n, j+n} = 0$.

Efficiently finding matrices $Z$ and $W$ which adhere to a given sparsity structure remains a challenging aspect of Algorithm \ref{n_iteration}. 
In the case of $k$-regular graphs, Tam shows that setting $Z=W$ and setting $Z_{ij} = \frac{2}{k}$ satisfies \eqref{oars}, but this construction does not extend to arbitrary graphs \cite{tam2023frugal}. 
The SDP design methods in \cite{bassett2024optimaldesignresolventsplitting} also accommodate a much wider range of non-regular graphs, with a few restrictions on their level of connectivity, and allow the selection of any convex objective over the matrices. 
These SDP methods require a centralized matrix design process, however, and for problems of even moderate size they can become a bottleneck which takes a large fraction of the total execution time.
Their efficiency can be significantly improved by setting $Z=W$, tightening the formulation, or otherwise limiting the problem scope, but the computational challenge remains for large problems.

In the case of $Z=W$, given the graph Laplacian interpretation of the matrices, the problem of finding a feasible set of matrix parameters adhering to a graph reduces to that of finding a connected weighted subgraph of the communication graph such that the weighted degree of each node equals 2.
Their equality directly satisfies equation \eqref{Z_succeq_W}.
Equations \eqref{null_W_1} and \eqref{sum_Z_1} are satisfied if $W$ (and $Z$) are the weighted graph Laplacians of a connected weighted graph, because every graph Laplacian satisfies $W\1=0$, and the dimension of the null space of a graph Laplacian is equal to the number of connected components of the graph, even in the weighted case \cite{fiedler1973algebraic}.
Being a subgraph then satisfies the adherence requirement, and having a weighted degree of two for each node satisfies requirement \eqref{diag_Z_2}.

\subsection{Efficient Matrix Selection}

As an alternative to the SDP methods introduced in \cite{bassett2024optimaldesignresolventsplitting} for finding $Z$ and $W$, we next show that these matrices can instead be formed efficiently in a decentralized manner using the Sinkhorn-Knopp algorithm \cite{sinkhorn1967concerning} (also known as matrix scaling, iterative proportional fitting procedure (IPFP), or the RAS method). 
The Sinkhorn-Knopp algorithm is an iterative algorithm for constructing matrices adhering to the sparsity pattern of an input matrix but with specified row and column sums.
In the doubly stochastic case where each row and column sum to one, given a non-negative square input matrix, the algorithm alternates between scaling rows and columns to have unit sum.
We use the notation $SK(A)$ to refer to the limit of the Sinkhorn-Knopp algorithm in this doubly stochastic case, when it exists, for a given matrix $A$.

The convergence of the Sinkhorn-Knopp algorithm rests on the concept of \textit{support}.
Given some non-negative square matrix $A \in \R_+^{n \times n}$, there is a corresponding
bipartite graph over $2n$ nodes with $A$ interpreted as a weighted bi-adjacency matrix, where $A_{ij}>0$ implies an edge between node $i$ in the first set of nodes and node $j$ in the second set of nodes.
We say $A$ has \textit{support} if there exists a perfect matching between the two sets in the partition over some set of edges in $A$.
We define $A_M \in \R_+^{n \times n}$ as the matrix with only the edges in $A$ which are part of a perfect matching.
Sinkhorn and Knopp, in \cite{sinkhorn1967concerning}, show that their algorithm converges if and only if $A$ has support, and that the sparsity pattern of $SK(A)$ is that of $A_M$.
Furthermore, they show that if $A$ is symmetric, $SK(A)$ will be symmetric as well.

The Sinkhorn-Knopp algorithm can also be applied in a decentralized manner.
Given an adjacency matrix $A$ and associated graph $G(A)$, we assume that each node of $G(A)$ is a computation node which can communicate with its neighbors in the graph.
We then find $SK(A)$ by iteratively scaling the edge weights to unit sum.
To do so, each node (in parallel) begins with an edge weight of 1 for each of its neighbors (corresponding with $A_{ij}=1$ in its adjacency matrix), with all other weights set to zero. 
It then scales these weights so that their sum equals one (this corresponds to the row scaling).
Next, the node transmits each updated edge weight to the corresponding neighboring node. 
Once it has received the edge weights from all its neighbors, it again scales the received values so that they sum to one (this corresponds to the column scaling).
Each node then transmits the newly adjusted edge weights back to its neighbors, repeating the process until convergence.

Applying Sinkhorn-Knopp to the problem of finding $W$ and $Z$ adhering to a given communication graph $G(A)$ leads to the following results:

\begin{theorem}\label{sk_frs}
    Let $A \in \{0, 1\}^{n \times n}$ be a symmetric matrix with support such that $\diag(A)=0$ and $G(A_M)$ is connected.
    Then matrices $W, Z \in \Sp^n$ defined as $W=Z=2\left(\I-SK\left(A\right)\right)$ satisfy \eqref{oars} and adhere to the sparsity structure of $G(A)$.
\end{theorem}
\begin{proof}
    By the results in \cite{sinkhorn1967concerning}, the Sinkhorn-Knopp algorithm on $A$ converges, and $SK(A)$ is a symmetric doubly-stochastic adjacency matrix with zeros on the diagonal and a sparsity pattern which matches $A_M$.
    $SK(A)$ is therefore the weighted adjacency matrix of a connected weighted graph.
    Therefore $Z$ and $W$ are well-defined, and $Z \succeq W$ \eqref{Z_succeq_W} is satisfied by their equality.
    $Z$ and $W$ are positive semi-definite by the Gershgorin Circle Theorem.
    Since $SK(A)\1 = \1$, we have $W\1=Z\1=2(\I-SK(A))\1 = 0$, so $\1^T Z \1 = 0$ and \eqref{sum_Z_1} is satisfied.
    The connectivity of the weighted graph with adjacency matrix $2SK(A)$ means that its graph Laplacian has null space of dimension 1, so $\nullspace{W} = \text{span}(\1)$ \eqref{null_W_1}.
    The requirement $\diag(Z)=2\1$ \eqref{diag_Z_2} is satisfied because $SK(A)$ has zeros on its diagonal, and $Z$ and $W$ therefore satisfy the requirements of \eqref{oars}.
    Furthermore, since $SK(A)$ has the same sparsity pattern as $A_M$, and $G(A_M)$ is a subgraph of $G(A)$, $SK(A)$ adheres to the sparsity pattern of $G(A)$.
\end{proof}
It may not be directly evident a priori whether the adjacency matrix $A$ of a given connected graph $G(A)$ has support and $A_M$ is connected.
In the case of the $2$-Block designs \eqref{2-block} above, however, we can directly construct a larger matrix $A' \in \{0,1\}^{2n \times 2n}$ which 2-Block adheres to the sparsity pattern of $G(A)$, has support, is connected, and has $A'_M = A'$, so that $A'_M$ is also connected.
We establish this result with the following corollary, which will be useful in the remainder:
\begin{corollary}\label{sk_ones}
    Let $A \in \{0, 1\}^{n \times n}$ be the adjacency matrix of a connected graph.
    Then matrix $A + \I$ has support, and matrices $W, Z \in \R^{2n \times 2n}$ defined as 
    \[W=Z=2\begin{bmatrix} \I & -SK(A + \I) \\ -SK(A + \I) & \I\end{bmatrix}\] satisfy \eqref{oars} and 2-Block adhere to the sparsity structure of $G(A)$.
\end{corollary}
\begin{proof}
    We first show that $A + \I$ has support, and that for a matrix $A' = \begin{bmatrix} 0 & A + \I \\ A + \I & 0\end{bmatrix}$, we have $SK(A') = \begin{bmatrix} 0 & SK(A + \I) \\ SK(A + \I) & 0\end{bmatrix}$.
    We then show that $A'$ satisfies the conditions of Theorem \ref{sk_frs}.
    We conclude by showing that $Z$ and $W$ formed from $SK(A')$ as given in Theorem \ref{sk_frs} 2-Block adhere to $G(A)$.

    Interpreting $A + \I$ as a bi-adjacency matrix, we can form a perfect matching over the nonzero diagonal entries, pairing node $i$ in the first set with node $i$ in the second set.
    We call this the diagonal matching.
    This establishes that $A + \I$ has support.

    We next establish that $A + \I = (A + \I)_M$, that is, all entries in $A + \I$ are part of some perfect matching.
    For all nonzero entries $A_{ij}$, we have entry $A_{ji}$ by symmetry, since $A$ is the adjacency matrix of an undirected graph. 
    We therefore have a perfect matching by replacing the matchings $(i, i)$ and $(j, j)$ in the diagonal matching with $(i, j)$ and $(j, i)$.
    These are the only nonzero entries in $A + \I$ other than those in the perfect matching on the diagonal, so each entry in $A + \I$ is part of a perfect matching, and $A + \I = (A + \I)_M$.
    Therefore $SK(A + \I)$ exists and is connected.

    Let $A' = \begin{bmatrix} 0 & A + \I \\ A + \I & 0\end{bmatrix}$.
    Direct inspection of the Sinkhorn-Knopp algorithm shows that if $B_k$ is the result after iteration $k$ of the Sinkhorn-Knopp algorithm applied to $A + \I$, then the result of iteration $k$ of the Sinkhorn-Knopp algorithm on $A'$ is $B'_k = \begin{bmatrix} 0 & B_k\\ B_k & 0\end{bmatrix}$.
    Therefore $SK(A') = \begin{bmatrix} 0 & SK(A + \I) \\ SK(A + \I) & 0\end{bmatrix}$.
    Using a similar substitution process to that above, it can also be shown that $A' = A'_M$, that is, every nonzero entry in $A'$ is part of some perfect matching.

    Furthermore, $G(A') = G(A'_M)$ is connected by the connectivity of $G(A)$ and the edges $(i, i+n)$ and $(i+n, i)$ for $i=1, \dots, n$. 
    This can be seen by noting that, since $G(A)$ is connected, it contains a path $p = (i, p_1, \dots, p_k, j)$ from any node $i$ to any node $j$. 
    In $G(A')$ we can use an expanded version of this path (called $p'$) to connect any nodes $i$ and $j$ by alternating between nodes in $\{1, \dots, n\}$ and those in $\{n+1, \dots 2n\}$.
    In the case where $i,j \leq n$, this expanded path is $p' = (i, p_1+n, p_1, p_2+n, p_2, \dots, p_k+n, p_k, j+n, j)$.
    If $i > n$, we add the link from $i$ to $i-n$ at the beginning.
    The case of $j > n$ is included above by using the path in $G(A)$ from $i$ to $j-n$ and terminating one step early.
    
    Therefore $A'$ is a symmetric matrix with support and $G(A'_M)$ is connected. 
    By Theorem \ref{sk_frs}, $W$ and $Z$ such that \[W=Z=2\left(\I-SK\left(A'\right)\right) = 2\begin{bmatrix} \I & -SK(A + \I) \\ -SK(A + \I) & \I\end{bmatrix}\] satisfy the requirements of \eqref{oars}.
    By construction, $A'_{i, j+n} = A'_{j, i+n} = A'_{j+n, i} = A'_{i+n, j} = 0$ for all $i \neq j$ such that $A_{ij} = 0$, and therefore, by Theorem \ref{sk_frs}, $W_{i, j+n} = Z_{i, j+n} = W_{i+n, j} = Z_{i+n, j} = 0$ for any nodes $i \neq j$ such that $A_{ij} = 0$, and $Z$ and $W$ 2-Block adhere to the sparsity structure in $G(A)$.
\end{proof}

\section{Noisy Sensor Network Localization via Proximal Splitting}\label{Sec:SNL}

\subsection{Proximal Splitting Design}
We now describe the proximal splitting over the node-based relaxation \eqref{convex_snl_nsdp}. 
We define $\delta_i$ as the indicator function on the positive semidefinite cone for principal submatrix $S^i$, and define $f_i$ for $i = 1, \dots, 2n$ so that the first $n$ functions refer to \eqref{g_def} and the second $n$ function refer to the indicator functions.
That is,
\begin{align}
    \delta_i(X,Y) &= \begin{cases}
        0 & S^i(X,Y) \succeq 0\\
        \infty & S^i(X,Y) \not\succeq 0
    \end{cases}\\
    f_i(X,Y) &= \begin{cases}
        g_i(X, Y) & i = 1, \dots, n \\
        \delta_{i-n}(X, Y) & i = n+1, \dots, 2n
    \end{cases}\label{f_def}
\end{align}
For any scaling parameter $\alpha > 0$, problem \eqref{convex_snl_nsdp} is then equivalent to 
\begin{equation}
    \min_{X,Y} \alpha \sum_{i}^{2n} f_i(X,Y)
\end{equation}
This objective function can be directly split into $2n$ functions, with $n$ 1-norm summations $g_i$ and $n$ indicator functions $\delta_i$. 
We therefore adopt a 2-Block design for the splitting.
We define the adjacency matrix $A \in \{0,1\}^{n \times n}$ based on the sets $\mathcal{N}_i$ as
\begin{equation}\label{A_def}
    A_{ij} = \begin{cases}
        1 & j \in \mathcal{N}_i \text{ or } i \in \mathcal{N}_{j}\\
        0 & \text{otherwise}
    \end{cases}
\end{equation}
We assume in the remainder that neighborhoods are chosen so that $G(A)$ is a connected graph.
Transferring data over the graph $G(A)$ then ensures that each neighbor $j \in \mathcal{N}_i$ of a given sensor $i$ receives updates from $i$, as do any neighbors $j$ such that $i \in \mathcal{N}_j$.
Using the results of Corollary \ref{sk_ones}, we therefore choose $Z, W \in \Sp^{2n \times 2n}$ such that
\begin{equation}\label{Z_W_SK}
    Z = W = 2\begin{bmatrix} \I & -SK(A+\I)\\ -SK(A+\I) & \I\end{bmatrix}.
\end{equation}

\subsection{Proximal Function Evaluation}
We assign each sensor $i$ the prox functions over $\delta_i$ and $g_i$.
The proximal functions over the indicator function $\delta_i$ can be computed as a projection onto the PSD cone, returning $V_+\Lambda_+ V_+^T$, where $\Lambda$ and $V$ are the eigenvalues and eigenvectors of the input matrix, and $\Lambda_+$ and $V_+$ return the non-negative eigenvalues and their associated eigenvectors, respectively.
We also need to find the prox of $g_i$, which under a weighted norm for $X$ (to align with the Frobenius norm on symmetric $S(X,Y)$ in \cite{wang2006further}) is defined as 
\begin{equation}\label{g_prox}
\argmin_{X, Y} g_i(X,Y) + \frac{1}{2}\norm{Y - Y^k}_F^2 + \norm{X - X^k}_F^2.
\end{equation}
We note that for $j, j' \notin \mathcal{N}_i$, the minimizer $(X^*, Y^*)$ in \eqref{g_prox} returns $X^*_{j \cdot} = X^k_{j \cdot}$, $Y^*_{jj} = Y_{jj}^k$, $Y^*_{ij} = Y^*_{ji} = Y_{ij}^k$, and $Y^*_{jj'} = Y_{jj'}^k$.
With appropriate vectorization of $(X,Y)$, their coefficients, and the constant terms, we can reduce \eqref{g_prox} to a least absolution deviation problem with least squares regularization.
For each $i\in \{1, \dots, n\}$, let $(j_1, \dots, j_{|\mathcal{N}_i|})$ enumerate the elements in $\mathcal{N}_i$ and $(k_1, \dots, k_{|\mathcal{M}_i|})$ enumerate the elements in $\mathcal{M}_i$.
We then vectorize as 
\begin{align}
\vvec_i(X,Y) &= \left(Y_{ii}, Y_{j_1 j_1}, Y_{i j_1}, \dots, Y_{j_{|\mathcal{N}_i|} j_{|\mathcal{N}_i|}}, Y_{i j_{|\mathcal{N}_i|}}, X_{i \cdot}\right) \in \R^{1 + 2|\mathcal{N}_i| + d} \\
c_i &= \left(d_{ij_1}^2, \dots, d_{ij_{|\mathcal{N}_i|}}^2, d_{ik_1}^2 - \norm{a_{k_1}}^2, \dots, d_{i{k_{|\mathcal{M}_i|}}}^2 - \|a_{k_{|\mathcal{M}_i|}}\|^2\right) \in \R^{|\mathcal{N}_i| + |\mathcal{M}_i|} \\
D_i &= \text{diag}(1,1,\sqrt{2},1,\sqrt{2}, \dots, 1, \sqrt{2}, \sqrt{2}, \dots, \sqrt{2}) \in \R^{\left(1+2|\mathcal{N}_i|+d\right) \times \left(1+2|\mathcal{N}_i|+d\right)} \\
N_i &= \begin{bmatrix}
1 & -2 & 0 &  0 & \dots & 0 & 0 \\
0 &  0 & 1 & -2 & \dots & 0 & 0 \\
  &    &   &    & \ddots &  &   \\
0 &  0 & 0 & 0  & \dots & 1 & -2 \\
\end{bmatrix} \in \R^{|\mathcal{N}_i|  \times 2|\mathcal{N}_i|}\\
M_i &= \begin{bmatrix}
    a_{k_1} \\
    a_{k_2} \\
    \vdots \\
    a_{k_{|\mathcal{M}_i|}}
\end{bmatrix} \in \R^{|\mathcal{M}_i| \times d} \\
K_i &= \begin{bmatrix}
    \1 & N_i & 0 \\
    \1 & 0   & M_i 
\end{bmatrix} \in \R^{\left(|\mathcal{N}_i| + |\mathcal{M}_i|\right) \times \left(1 + 2|\mathcal{N}_i| + d\right)}
\end{align} 
Here $D_i$ doubles the entries in $S^i(X,Y)$ which are off-diagonal (and therefore occur twice in the matrix but only once in the vector) and $K_i\vvec_i(X,Y)$ returns each entry of $X$ and $Y$ in the absolute value summation with the appropriate coefficients.
We define the vectorized decision variable $w_i \in \R^{1 + 2|\mathcal{N}_i| + d}$ as $w_i = \vvec_i(X,Y)-\vvec_i(X^k, Y^k)$ and $c_i^k = c_i + K_i\vvec_i(X^k,Y^k)$, and then have the following problem:
$$\min_{w_i \in \R^{1 + 2|\mathcal{N}_i| + d}} \norm{c_i^k - K_i w_i}_1 + \frac{1}{2}\norm{D_i w_i}^2.$$
This is a regularized least absolute deviation problem which can be solved with a variety of methods.
In this case, we solve it with ADMM, which allows us to warm start subsequent solves using the previous solution vectors.
We do so with the following reformulation:
\begin{align}
    \min_{y_i \in \R^{|\mathcal{N}_i| + |\mathcal{M}_i|}, w_i \in \R^{1 + 2|\mathcal{N}_i| + d}} \quad & \norm{y_i}_1 + \frac{1}{2}\norm{D_i w_i}^2 \\
    \text{s.t.} \quad & y_i = c_i^k - K_i w_i
\end{align}
For some starting values $\lambda_i^0, y_i^0 \in \R^{|\mathcal{N}_i| + |\mathcal{M}_i|}$, $\nu=1$, and $\rho > 0$, the ADMM execution for each sensor $i$ is then:
\begin{align}
    w_i^{\nu} &= \left(\rho K_i^T K_i + D_i^T D_i\right)^{-1} \rho K_i^T\left(\lambda_i^{{\nu}-1} + c_i^k - y_i^{{\nu}-1}\right) \label{cho_step}\\
    y_i^{\nu} &= \text{soft}_{1/\rho}\left(c_i^k - K_i w_i^{\nu} + \lambda_i^{{\nu}-1}\right) \\
    \lambda_i^{\nu} &= \lambda_i^{{\nu}-1} - y_i^{\nu} - K_i w_i^{\nu} + c_i^k\label{lambda_update}
\end{align}
where $\text{soft}_{\tau}(\cdot)$ is the soft-thresholding operator defined as $\text{soft}_{\tau}(\cdot) = \text{sign}(\cdot)\left(|\cdot|-\tau\right)_{+}$.
We can solve \eqref{cho_step} efficiently by computing the Cholesky factorization of $\rho K^T K + D^T D$ once at the beginning of the algorithm. 
We terminate the ADMM algorithm for the prox when the Euclidean norm of the $\lambda$ update in \eqref{lambda_update} falls below a given tolerance, which we tighten as the algorithm progresses.

\subsection{Decentralized ADMM}
We compare the decentralized 2-Block matrix parametrized algorithm with a decentralized ADMM algorithm which splits over the same $2n$ functions $f_i$ in \eqref{f_def}. 
We compute the prox of $f_i$ and $f_{i+n}$ in series on each sensor $i$, with all sensors acting in parallel. 
The required communications then take place between neighbors prior to the next iteration.

Using the same adjacency matrix $A$ as in \eqref{A_def} we form the graph $G(A')$ over $2n$ nodes with the adjacency matrix $A' = \begin{bmatrix}
    A & A + \I \\
    A + \I & A \\
\end{bmatrix}.$ 
Let $\mathcal{E}'$ be the set of edges $(i,j)$ in $G(A')$.
We define $\mathcal{K}_i$ as the set of neighbors of node $i$ in the graph $G(A')$, noting that $\mathcal{K}_i$ is defined over $i=1, \dots, 2n$. 
For $i=1, \dots n$ we therefore have 
\begin{align}
\mathcal{K}_i &= \mathcal{N}_i \cup \{n+i\} \cup \left\{j+n | j \in \mathcal{N}_i\right\} \\
\mathcal{K}_{i+n} &= \mathcal{N}_i \cup \{i\} \cup \left\{j+n | j \in \mathcal{N}_i\right\}
\end{align}
Then with the additional variable $S_{ij} \in \R^{(d+n)\times(d+n)}$ for each edge in $\mathcal{E}'$, which we use to force equality between nodes $i$ and $j$, we form the equivalent problem
\begin{align}
    \label{convex_snl_nsdp_graph}
    \minimize_{\substack{\{X_i, Y_i \}_{ i = 1 \dots 2n}\\ \{S_{ij}\}_{(i,j) \in \mathcal{E'}}}}  \quad &\alpha \sum_{i=1}^{2n} f_i(X_i, Y_i) \\
    \text{s.t.} \quad & \begin{rcases*}
        S(X_i, Y_i) = S_{ij} \\
        S(X_j, Y_j) = S_{ij}
    \end{rcases*} \quad \forall (i,j) \in \mathcal{E'}
\end{align}
We then use the derivation in \cite{ryu2022large} to reduce \eqref{convex_snl_nsdp_graph} to the following decentralized ADMM algorithm:
\begin{align}
    U^{k+1}_i &= \prox_{\frac{\alpha}{|\mathcal{K}_i|}f_i}\left(V^k_i\right) \quad \forall i \in \mathcal{N} \\
    R^{k+1}_i &= \frac{1}{|\mathcal{K}_i|}\sum_{j \in \mathcal{K}_i} U_j^{k+1} \quad \forall i \in \mathcal{N} \\
    V^{k+1}_i &= V^k_i + R^{k+1}_i - \frac{1}{2}R^k_i - \frac{1}{2}U^k_i 
\end{align}
The prox functions are computed identically to those in Algorithm \ref{n_iteration}, giving both approaches the same computational complexity for their subproblems. %
The matrix parametrized algorithm does, however, have an additional round of communication in each iteration to transmit the data from sensors to their neighbors for the computation of $\v{L}\v{x}$ in \eqref{n_itr1} for the second block of $n$ prox functions, although the total number of bytes transmitted in each iteration remains the same.
We note that the ADMM memory requirement for each sensor, with the calculation of both $f_i$ and $f_{i+n}$ and excluding the requirement for the proximal subproblems, is that of six floating point $n+d$ symmetric matrices ($V$, $R$, and $U$ for $i$ and $i+n$).
The memory required for Algorithm \ref{n_iteration} is one third smaller, at four floating point $n+d$ symmetric matrices ($v$ and $x$ for $i$ and $i+n$).

\section{Numerical Experiments}\label{Sec:Experiments}

We begin by comparing the performance of Algorithm \ref{n_iteration} with decentralized ADMM using both cold start and warm start values. 
We then compare the performance of the matrix parameters $W$ and $Z$ constructed using the Sinkhorn-Knopp algorithm to those constructed via the SDP methods in \cite{bassett2024optimaldesignresolventsplitting}, evaluating both the time required to generate the matrix parameters and the rate of convergence of Algorithm \ref{n_iteration} using those parameters.
We conclude by investigating the use of the objective value of \eqref{convex_snl_nsdp} as an early stopping criterion reducing the required number of iterations of Algorithm \ref{n_iteration}, and as a method to produce improved location estimates. We propose a criterion which outperforms the solution to \eqref{convex_snl_nsdp} in over 60 \% of our experiments.
Unless otherwise noted, matrix parameters $Z$ and $W$ are 2-Block matrices selected using the Sinkhorn-Knopp algorithm as given in \eqref{Z_W_SK}.

We create problem data as in \cite{wang2006further}, generating $n=30$ sensor locations and $m=6$ anchor locations in $[0,1]^2$ randomly according to a uniform distribution. 
To build the neighbors $\mathcal{N}_{i}$ for each sensor, we select sensors with distance less than 0.7 from sensor $i$, up to a maximum of 7 (which we then select at random uniformly from the set of sensors within range).
The noisy distance data is given by $d_{ij} = d_{ij}^0(1+0.05 \varepsilon_{ij})$ and $d_{ik} = d_{ik}^0(1+0.05 \varepsilon_{ik})$ where each $\varepsilon_{ij}$ ($\varepsilon_{ik}$) is a standard Gaussian random variable generated independently for each true distance $d^0_{ij}$ ($d^0_{ik}$) and 0.05 is the noise factor. 
Algorithm \ref{n_iteration} uses a step size of $\gamma=0.999$ and a scaling parameter of $\alpha=10.0$.
ADMM uses a scaling parameter of $\alpha=150.0$.
Both scaling parameters which are independently hand-tuned to optimize performance.
The ADMM scaling closely matches the matrix-parametrized scaling multiplied by $\frac{1}{2n}\sum_{i=1}^{2n}|\mathcal{K}_i|$, which aligns with the observation that $\frac{1}{|\mathcal{K}_i|}$ is the additional scaling parameter for each prox in the decentralized graph ADMM. 
Where noted, we calculate reference solutions ($\bar{X}$) to the relaxation \eqref{convex_snl_nsdp} using MOSEK \cite{mosek}.
For reproducibility, the code supporting these experiments is available in a \href{https://github.com/peterbarkley/snl/}{publicly accessible repository on Github} \cite{snl-repo}.

\begin{figure}[h!]
    \centering
    \begin{subfigure}{.49\textwidth}
    \centering        
    \includegraphics[width=\textwidth]{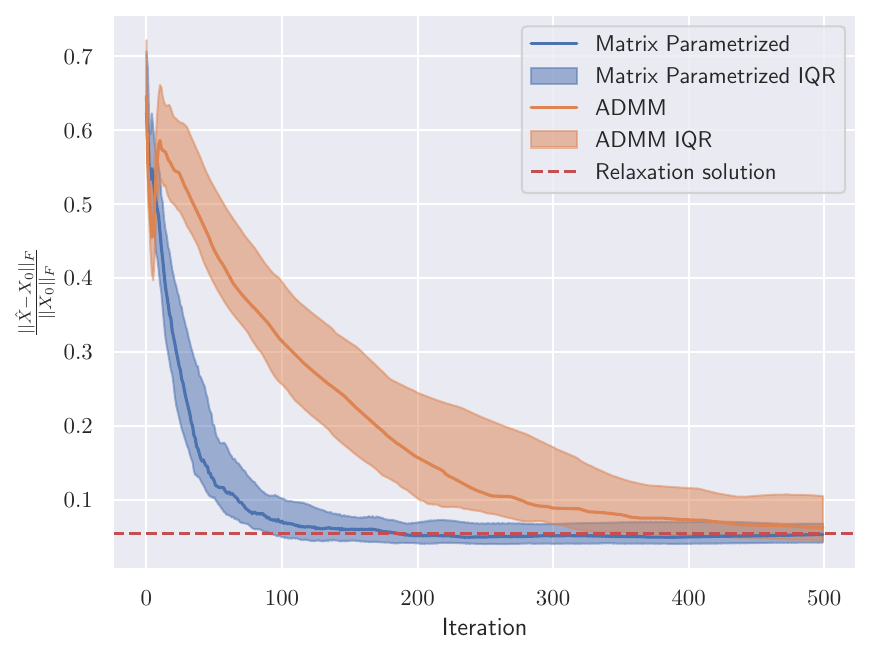}
    \caption{Cold Start}
    \label{fig:cold_admm}
\end{subfigure}%
\hspace{.019\textwidth}%
\begin{subfigure}{.49\textwidth}
    \centering        
    \includegraphics[width=\textwidth]{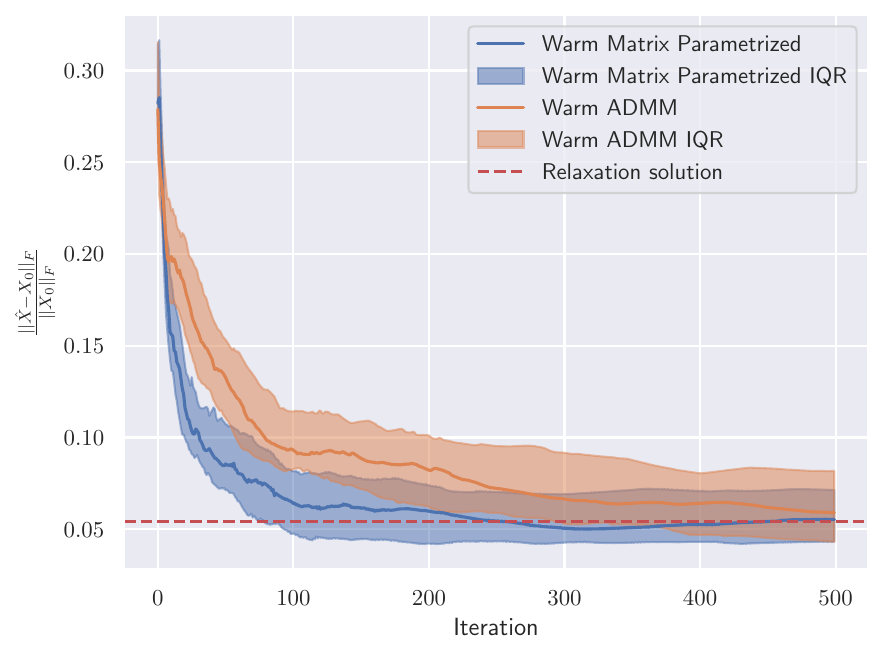}
    \caption{Warm Start}
    \label{fig:warm_admm}
\end{subfigure}
\caption{Relative error from the true location for Algorithm \ref{n_iteration} and ADMM}
\label{fig:oars_admm}
\end{figure}

Figure \ref{fig:oars_admm} depicts the convergence of Algorithm \ref{n_iteration} and ADMM over 50 sets of randomized data.
The cold start experiments, starting with an initialization at zero, are given by Figure \ref{fig:cold_admm}, and those starting at the warm start value (described below) are given by Figure \ref{fig:warm_admm}.
In these experiments, we denote by $\hat{X}^k_i$ the iteration $k$ value of the estimate of the sensor $i$ location held at sensor $i$, and $X_0$ the array of true location values. 
The relative error is then given as $\frac{\norm{\hat{X}-X_0}_F}{\norm{X_0}_F}$. 
For each method, we display the median values of the relative error over the tests with their interquartile range (IQR).
Algorithm \ref{n_iteration} dominates the decentralized ADMM splitting, achieving distance errors at each iteration which are less than half of that of ADMM in the iterations prior to matching the relaxation solution distance error.
We also see that, on average, Algorithm \ref{n_iteration} reaches parity with the relaxation solution in less than 200 iterations.
This is the case even though the violation of the PSD constraints in \eqref{convex_snl_nsdp} has not yet converged fully to zero.
We observe similar results when varying the noise factor, the sensor neighborhood radius, the number of anchors, and the problem size.

Figure \ref{fig:warm_admm} shows that warm starting the algorithms with an estimate for $X$ can speed up the convergence of the algorithms. 
This warm start data could come, for example, from cooperative mobile sensors with an estimated velocity vector for each sensor, a previous location estimate, and a timestamp.
We simulate this by adding independent zero-mean Gaussian noise with standard deviation 0.2 to each component of the true displacement vector for each sensor and using this perturbed $\tilde{X}$ to warm start both Algorithm \ref{n_iteration} and ADMM. 
For the Algorithm \ref{n_iteration} warm start, we use the fact that at convergence we have $\v{v}^* = \v{u}^* + (\I - \v{L})\1\otimes x^*$ \cite{bassett2024optimaldesignresolventsplitting}.
With no knowledge of optimal subgradients $\v{u}^*$, we set them equal to 0.
The 2-Block nature of $Z$ means each row of the upper block of $L$ is 0, and each row of the lower block of $L$ sums to two.
This implies a warm start value of $v_i^0 = (\tilde{X}, \tilde{X}\tilde{X}^T)$ for $i = 1, \dots, n$, and $v_i^0 = -(\tilde{X}, \tilde{X}\tilde{X}^T)$ for $i = n+1, \dots, 2n$. 
These values satisfy the zero-sum condition on $\v{v}^0$ in Algorithm \ref{n_iteration}. 
We warm start ADMM similarly, by letting $V_i^0 = R_i^0 = U_i^0 = (\tilde{X}, \tilde{X} \tilde{X}^T)$ for all $i$.
When warm starting both Algorithm \ref{n_iteration} and ADMM, Figure \ref{fig:oars_admm} demonstrates that both algorithms take fewer iterations to converge. As in the cold start setting, Algorithm 1 consistently outperforms ADMM.

\subsection{Splitting Matrix Design Comparison}
We next analyze the performance of matrix parameters computed using the Sinkhorn-Knopp algorithm relative to the parameters computed using the SDP objectives proposed in \cite{bassett2024optimaldesignresolventsplitting}. 
We consider matrices restricted to the 2-Block design over the available communication graph and then optimize to either maximize the Fiedler value (also known as the algebraic connectivity), minimize the total effective resistance, or minimize the second-largest eigenvalue magnitude (SLEM) of the resulting weighted graph, as described in \cite{bassett2024optimaldesignresolventsplitting}.

\begin{figure}[!ht]
    \centering
    \begin{subfigure}{.49\textwidth}
    \centering        
    \includegraphics[height=5cm]{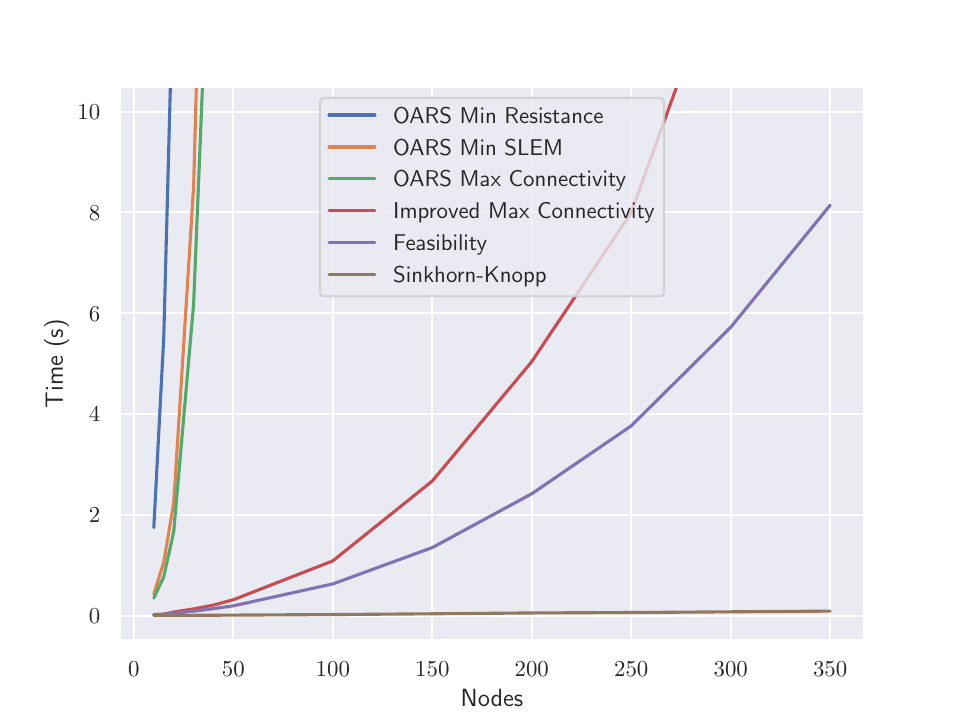}
    \caption{Mean matrix design computation time}
    \label{fig:methods}
\end{subfigure}%
\hspace{.019\textwidth}%
\begin{subfigure}{.49\textwidth}
    \centering        
    \includegraphics[height=4.7cm]{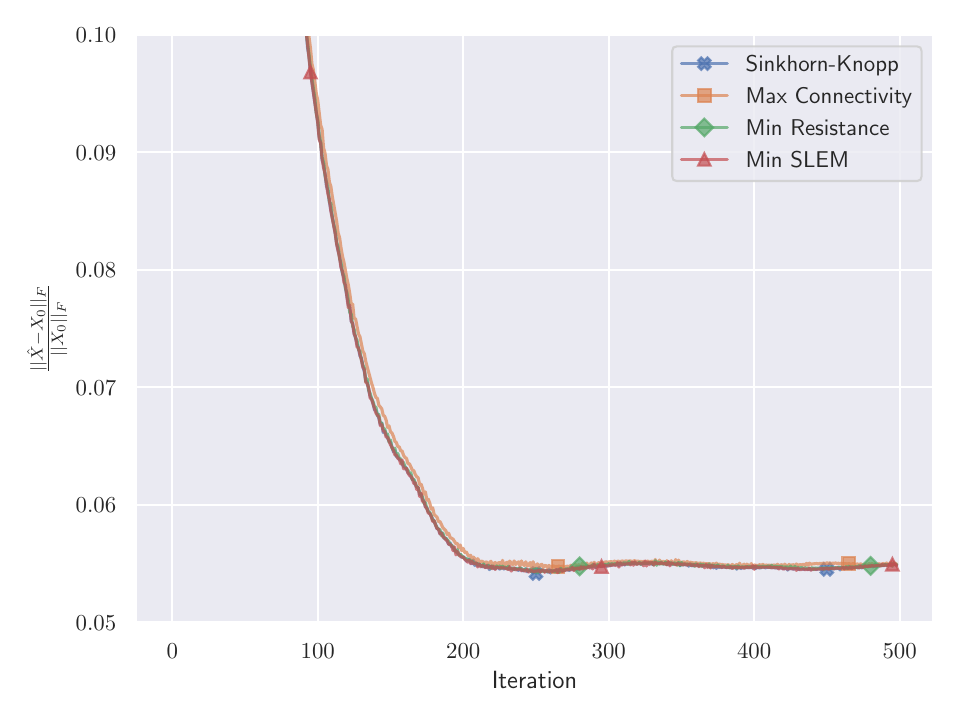}
    \caption{Relative error}
    \label{fig:matrices}
\end{subfigure}
\caption{Matrix design comparison}
\end{figure}
Figure \ref{fig:methods} provides the matrix design computation time for the Sinkhorn-Knopp algorithm, the OARS implementation of \cite{bassett2024optimaldesignresolventsplitting} provided in \cite{oars}, and our improved maximum Fiedler value SDP \eqref{maxfiedler} and feasibility SDP \eqref{feasibility}, which are described in the appendix in section \ref{improved_sdp}.
The Sinkhorn-Knopp algorithm scales to large problems without issue, whereas the time to compute the matrix parameters increases rapidly for the SDP methods, even in the improved maximum Fiedler value and feasibility formulations.
We also evaluate the convergence of the splitting algorithm under the various matrix designs, finding no significant difference between the performance of the different matrices as shown in Figure \ref{fig:matrices}.
The Sinkhorn-Knopp algorithm parameters therefore provide a convergence rate on \eqref{convex_snl_nsdp} which is indistinguishable from those of the other design approaches while scaling efficiently to large problem sizes.

\subsection{Early Termination}
We now shift to analyzing the performance of Algorithm \ref{n_iteration} when terminated prior to convergence to the solution of the relaxation \eqref{convex_snl_nsdp}.
In \cite{biswas2006semidefinite}, Biswas et al. note the tendency for the solutions of the SDP relaxation to exhibit ``crowding'', or a bias towards the center of mass of the anchors. 
We observe this fact in our tests, but also notice that the iterates of our first-order algorithm exhibit less centrality than the optimal value of \eqref{convex_snl_nsdp} returned by MOSEK.
In fact, terminating the algorithm early consistently provides slightly better estimates than $\bar{X}$.
Our observation suggests this is a result of the fact that the PSD constraints are only partially satisfied at early termination.
\begin{figure}[h!]
    \centering
    \begin{subfigure}{.49\textwidth}
    \centering        
    \includegraphics[width=.78\textwidth]{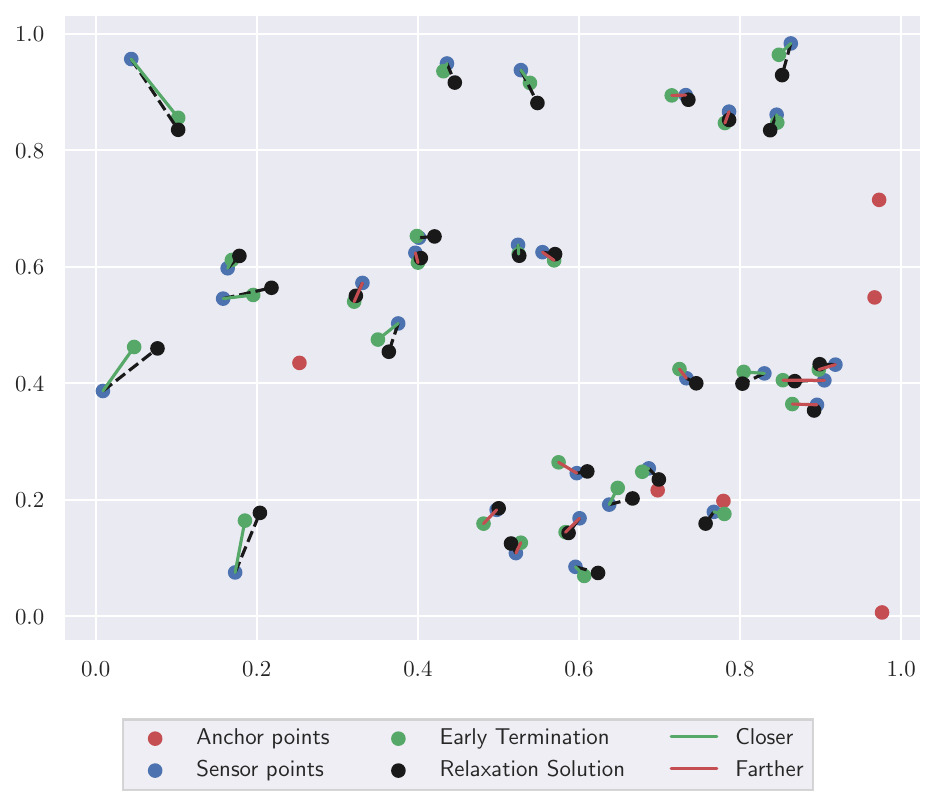}
    \caption{Early termination locations}
    \label{fig:central_plot}
\end{subfigure}%
\hspace{.019\textwidth}%
\begin{subfigure}{.49\textwidth}
    \centering        
    \includegraphics[width=\textwidth]{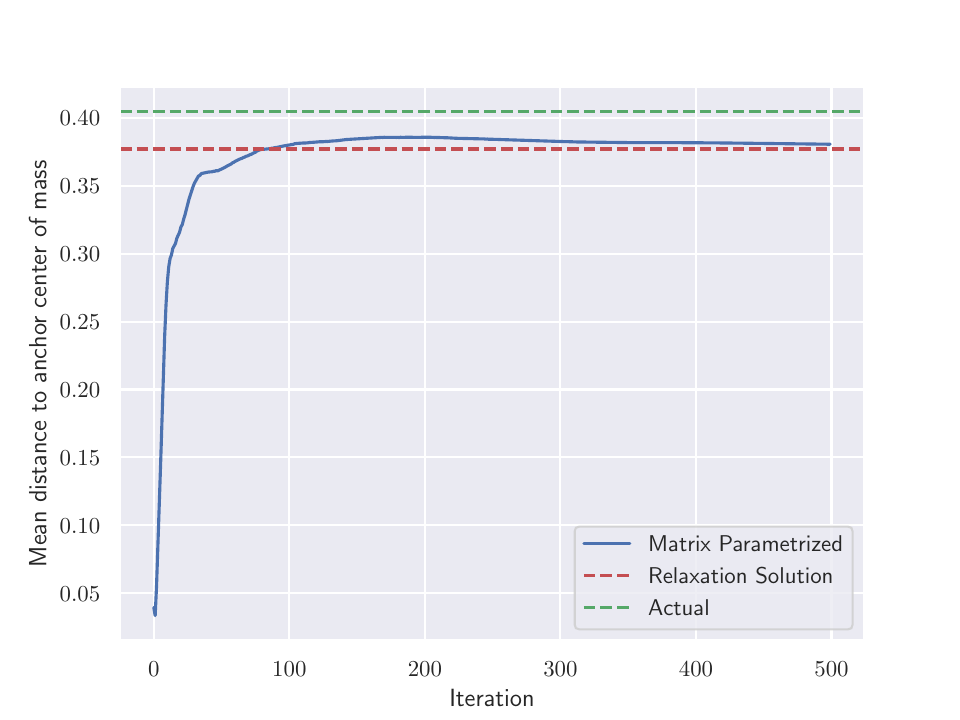}
    \caption{Mean central tendency of early solution}
    \label{fig:mean_central_plot}
\end{subfigure}
\caption{Early termination solution centrality}
\end{figure}

Figure \ref{fig:central_plot} illustrates the benefit of early termination graphically, depicting the true locations and estimated locations of sensors and anchors in a sample experiment.
The black dashed lines connect the true sensor locations (in blue) and the locations given by the relaxation solution $\bar{X}$ (in black).
The solid lines connect the true locations and the locations given by early termination of Algorithm \ref{n_iteration} (in green). 
We use solid green lines when the early termination location is closer than the relaxation solution, and red when they are not.
The figure shows the tendency of the early termination locations to be closer to the true location, and less biased toward the center.

Figure \ref{fig:mean_central_plot} shows the mean value over 20 trials of the centrality of the Algorithm \ref{n_iteration} location values at each iteration relative to the centrality of the true locations and that of $\bar{X}$, the solution to \eqref{convex_snl_nsdp}.
We define centrality here as $\frac{1}{n}\sum_{i=1}^n\|\hat{X}_i - \bar{a}\|$ where $\bar{a} = \frac{1}{m}\sum_{k=1}^m a_k$. 
This chart shows the tendency of $\bar{X}$ to concentrate nearer the anchors' center of mass than the true locations, as well as the iterates approaching the true locations before receding back to the solution of the relaxation.

We examine a variety of possible early stopping criteria to take advantage of the tendency of the pre-convergence iterates to be closer to the true locations than the final value.
We find that the minimum objective value (before the increasing weight of the PSD constraints forces it higher) provides a strong indication that we have reached the iterate closest to the true locations.
The maximum value of the centrality and the minimum value of $\frac{1}{n}\sum_{i=1}^n \hat{Y}_{ii} - \|\hat{X}_i\|^2$ (which \cite{biswas2004semidefinite} notes has an interpretation as the mean variance of the norm of the sensor locations taken as a random variable) also show correlation with the ideal stopping condition, but not unambiguously enough to make them good stopping criteria by themselves.
The magnitude of the changes in $\hat{X}_i$, $\hat{Y}_i$, and $v_i$ did not offer a strong signal for identifying the minimum distance iteration.
Given the tendency to underperform if the termination is too early, we track the objective value at each iteration, and terminate once the last 100 iterations have been higher that the lowest objective value observed to that point. 
Figure \ref{fig:early_stopping_histogram} shows a density estimate for the distribution of the mean distance of each sensor from its true location ($\frac{1}{n}\sum_{i=1}^n||\hat{X}_i - X_{0i}||_2$) across 300 tests of the objective value early termination criterion relative to the solution $\bar{X}$, showing that early stopping not only reduces the number of required iterations of Algorithm \ref{n_iteration}, but does so while slightly reducing the estimation error.
Figure \ref{fig:paired_early_stopping_histogram} shows a histogram of the paired difference between the mean distances from the true locations of the early termination iterate and the relaxation solution, showing the early termination solution is closer to the true locations in 64\% of cases.

\begin{figure}[h!]
    \centering
    \begin{subfigure}[T]{.49\textwidth}
    \centering        
    \includegraphics[width=\textwidth]{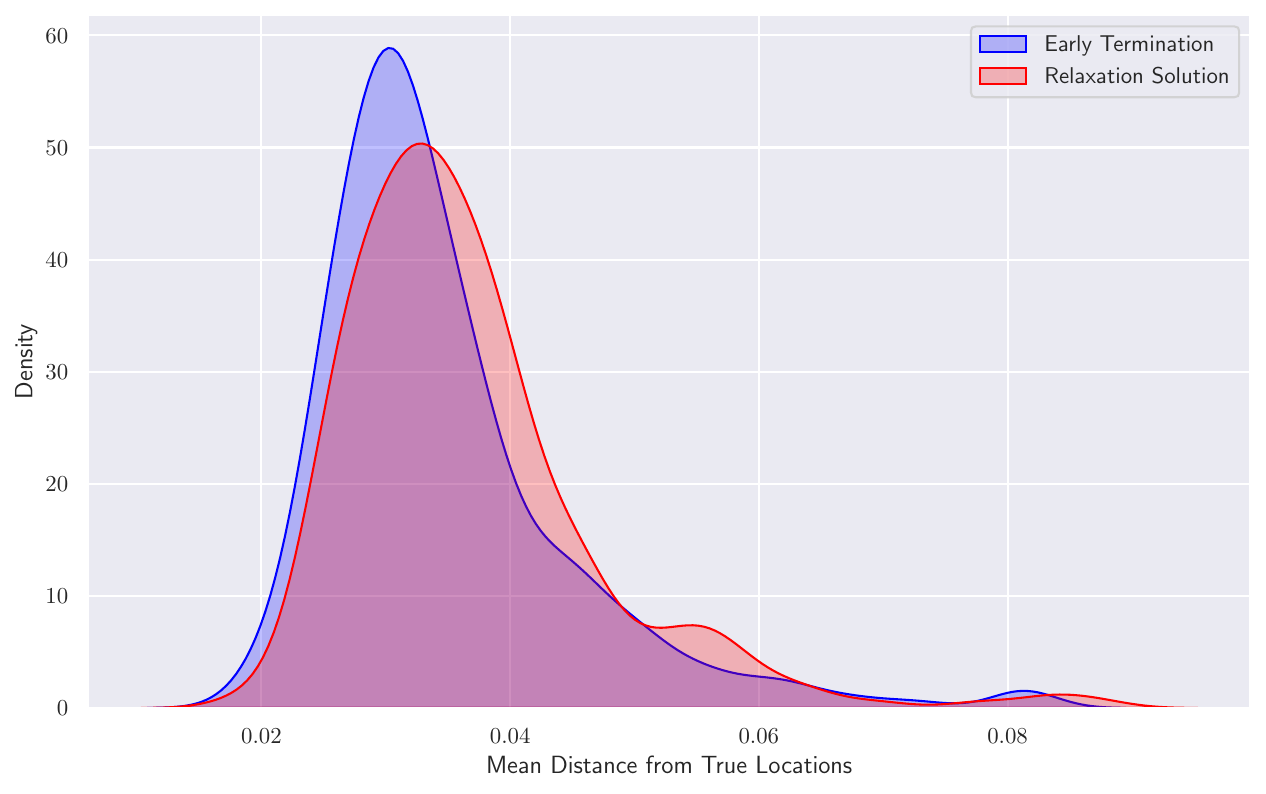}
    \caption{Early termination and IP mean distances}
    \label{fig:early_stopping_histogram}
\end{subfigure}%
\hspace{.01\textwidth}%
\begin{subfigure}[T]{.49\textwidth}
    \centering        
    \includegraphics[width=\textwidth]{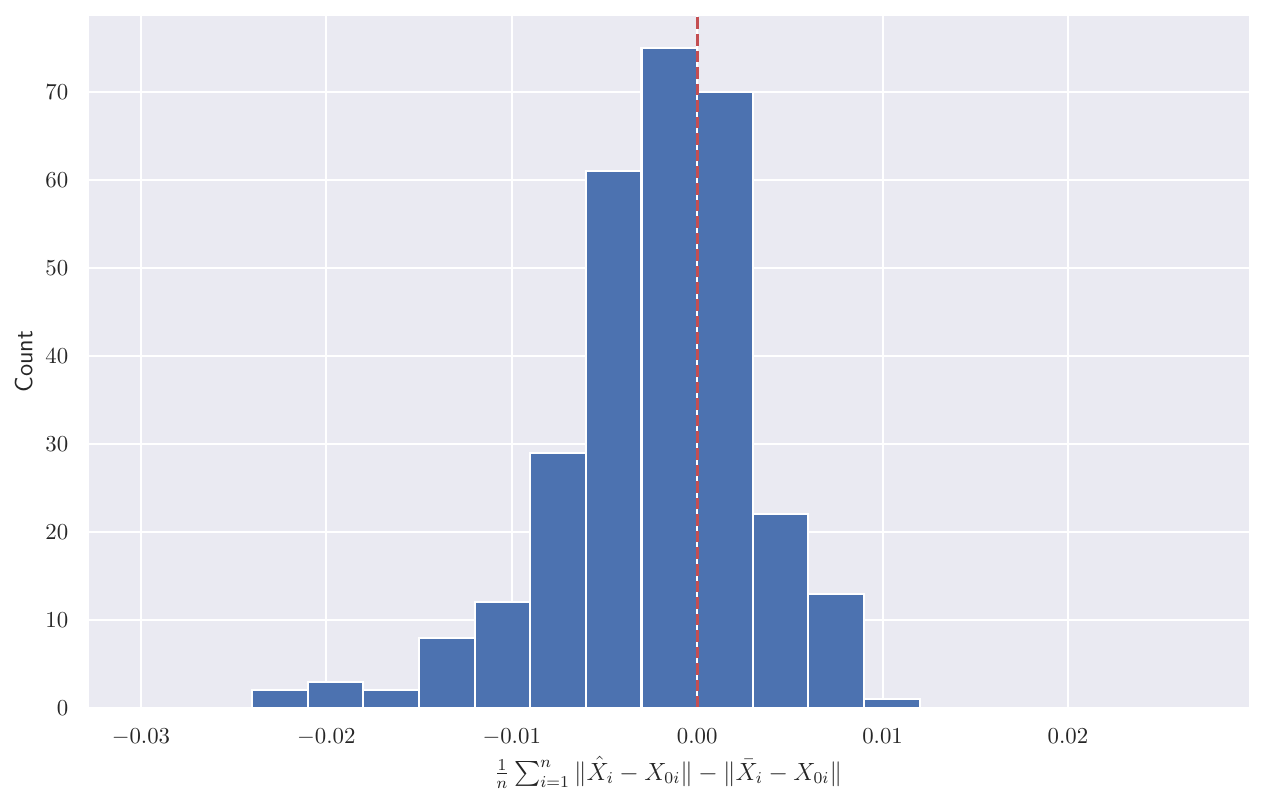}
    \caption{Paired differences}
    \label{fig:paired_early_stopping_histogram}
\end{subfigure}
\caption{Early termination performance}
\label{fig:early_term}
\end{figure}

\section{Conclusion}

This work demonstrates the efficacy and practical application of the decentralized matrix parametrized proximal splitting algorithm \cite{bassett2024optimaldesignresolventsplitting} on the noisy sensor network localization problem.
We present a novel efficient and decentralized approach to matrix parameter selection through the use of the Sinkhorn-Knopp algorithm, and show that, with these parameters, the matrix parametrized proximal splitting algorithm has performance similar to the much more expensive SDP-based matrix selection approaches in \cite{bassett2024optimaldesignresolventsplitting}. 
This innovation extends the applicability of the matrix parametrized proximal splitting algorithm to decentralized environments which do not have communication structures known in advance.
We compare the convergence of our decentralized splitting algorithm to decentralized ADMM algorithm over the same communication structure, demonstrating improved performance in cold start settings and when warm starting the algorithm using an a-priori estimate of sensor positions. Finally, we observe that in many cases, the iterates for the matrix parametrized proximal splitting drive the location estimates closer to the true location than the estimate provided at the relaxation's optimal value, noting that only approximate satisfaction of the PSD constraints in the node-based relaxation of the sensor network localization problem is required to achieve the minimal distance from true locations. To take advantage of this observation, we propose an early stopping criterion which produces a more accurate estimator of the sensor locations than the solution to the node-based relaxation of the sensor network localization problem.

\textbf{Acknowledgements} Both authors acknowledge support from Office of Naval Research awards N0001425WX00069 and N0001425GI01512.

\bibliography{files/thebib}%

\appendix 
\section{Appendix}\label{Appendix}
\subsection{Improved Matrix Design SDP}\label{improved_sdp}
We introduce two improved SDP formulations for matrix parameter selection.
The first improves the maximum Fiedler value formulation in \cite{bassett2024optimaldesignresolventsplitting}.
The second reduces this to a feasibility problem.
Our improved formulations rely on three primary changes.
First, setting $Z=W$ and removing the $W$ variable, thereby satisfying \eqref{Z_succeq_W} ($Z \succeq W$) directly.
Second, creating edge weight variables only over available edges and using the adjacency matrices to create a $Z$ matrix which satisfies \eqref{sum_Z_1} ($Z\1=0$) and adherence to the graph by construction.
And third, by leveraging our knowledge of the smallest eigenvector of $Z$ to find the second-smallest eigenvalue (the Fiedler value) of $Z$ without resorting to additional PSD constraints and variables. 
This allows us to both satisfy the null space condition \eqref{null_W_1} and maximize the Fiedler value.

The authors of \cite{bassett2024optimaldesignresolventsplitting} provide the following formulation for the maximum Fiedler value problem:
\begin{subequations}\label{main_prob}
    \begin{align}
    \max_{W, Z, c} \quad &c \label{obj}\\
    \mathrm{subject\, to} \quad& W \1 = 0 \label{con1}\\
    & \lambda_{1}(W) + \lambda_{2}(W) \geq c \label{con2}\\
    & Z - W \succeq 0 \label{con3}\\
    & \1^{T} Z \1 = 0  \label{con4}\\
    & \diag(Z) = 2 \label{con6}\\
    & (W, Z) \in \mathcal{C} \label{con7}\\
    & W \in \Sp^n, \; Z \in \mathbb{S}^{n}. \label{con8}
    \end{align}
\end{subequations}
In this formulation \eqref{main_prob}, we have a constraint set in \eqref{con7} to enforce the adherence to the graph, and constraints \eqref{con1} and \eqref{con2} together enforce the null space condition on $W$ \eqref{null_W_1} and provide the Fiedler value ($c$) for maximization.
In \cite{bassett2024optimaldesignresolventsplitting}, the authors note that constraint \eqref{con2} can be implemented with auxiliary variables $s \geq 0$ and $Y \in \mathbb{S}_{+}^{n}$ and the additional constraints
\begin{align}
W + Y - s \I &\succeq 0\\
2 s - \mathrm{Tr}(Y) & \geq c.
\end{align}
We show, however, that these auxiliary variables are unnecessary given the structure of $W$.

We begin with the shift to an edge weight variable construction for $Z$, following the framework in \cite{boyd2006convex}.
We define $E \in \R^{d \times n}$ with rows $e_{ij} = e_i - e_j$ for each feasible edge $(i,j)$ and $\v{z} \in \R_+^d$ as a decision variable for the edge weights.
The matrix $Z = E^T \v{z} E$ then has $Z_{ij} = -z_k$ where $z_k$ is the edge weight for the edge between nodes $i$ and $j$ if it exists in the graph, otherwise 0.
Entries $Z_{ii}$ total the sum of the edge weights for node $i$.
We know that the eigenvalue associated with the ones vector is therefore 0.
We also know that $\sum_{i=1}^n \lambda_i(Z) = 2n$, so the largest possible value for $\lambda_2$ must be no larger than $\frac{2n}{n-1}$. 
For all $n \geq 2$, we therefore have $\lambda_2 \leq 4$.
Therefore $\lambda_1(Z + 4 \1\1^T) = \lambda_2(Z)$.

Constraint \eqref{lambda2} in maximization problem \eqref{maxfiedler} therefore sets $c$ as the second-smallest eigenvalue of $Z$ and maximizes it.
Whenever a connected graph is feasible, problem \eqref{maxfiedler} with constraint \eqref{lambda2} will select the graph Laplacian of a connected graph.
The maximum Fiedler value objective value can then be formulated as:
\begin{subequations}
\label{maxfiedler}
\begin{align}
    \max_{\v{z}, c} & \quad c \nonumber \\
    \text{s.t.} & \quad E^T \v{z} E + 4 \1\1^T \succeq c\I \label{lambda2}\\
    & \quad \diag( E^T \v{z} E) = 2\1.
\end{align}
\end{subequations}
If we fix $c = 2(1-\cos{\frac{\pi}{n}})$ as in \cite{bassett2024optimaldesignresolventsplitting}, we can modify \eqref{maxfiedler} into an even simpler feasibility problem.
\begin{align}
    \label{feasibility}
    \max_{\v{z}} & \quad 0 \nonumber \\
    \text{s.t.} & \quad E^T \v{z} E + 4 \1\1^T \succeq c\I  \\
    & \quad \diag( E^T \v{z} E) = 2\1.
\end{align}
In both \eqref{maxfiedler} and \eqref{feasibility}, we then return $Z = W = E^T \v{z} E$.

\end{document}